\documentclass[12pt]{article}
\usepackage[T2A]{fontenc}
\usepackage[english]{babel}%
\usepackage{amsmath,amssymb,amsthm}
\usepackage{graphics}
\usepackage{graphicx}
\renewcommand{\leq}{\leqslant}
\renewcommand{\geq}{\geqslant}

\newtheorem{proposition}{\indent Proposition}

\begin{document}
\centerline{\large\bf A ``converse'' stability condition is necessary}
\smallskip
\centerline{\large\bf for a compact higher order scheme on non-uniform meshes}
\smallskip
\centerline{\large\bf for the time-dependent Schr\"{o}dinger equation}
\begin{center}{ Alexander Zlotnik$^{\dagger}$, Raimondas \v{C}iegis$^{\ddagger}$}
\end{center}
\par\noindent$^{\dagger}$
{National Research University Higher School of Economics, Myasnitskaya 20, 101000 Moscow, Russia,
e-mail: \text{azlotnik@hse.ru}
\smallskip\par\noindent
$^\ddagger$ {Vilnius Gediminas Technical University,
Saul{\. e}tekio al. 11, LT-10223 Vilnius, Lithuania,\\
e-mail: \text{raimondas.ciegis@vgtu.lt}}

\begin{abstract}
\noindent The stability bounds and error estimates for a compact higher order Numerov-Crank-Nicolson scheme on non-uniform space meshes for the 1D time-dependent Schr\"{o}dinger equation have been recently derived.
This analysis has been done in $L^2$ and $H^1$ mesh norms and used the non-standard ``converse'' condition $h_\omega\leq c_0\tau$, where $h_\omega$ is the mean space step, $\tau$ is the time step and $c_0>0$.
Now we prove that such condition is \textit{necessary} for some families of non-uniform meshes and any space norm.
Also numerical results show unacceptably wrong behavior of numerical solutions (their dramatic mass non-conservation) when this condition is violated.
\end{abstract}
\section{\large Introduction}
Compact higher order Numerov-type three-point finite-difference schemes are widely used in practical numerical solution of the Schr\"odinger-like equation; for example, see recent papers \cite{CGL12,HBG13,Ro14,S09,SPM10,SS14}, etc.
This is due to their higher order in space and ability to reduce errors significantly, with almost the same cost of implementation as for more standard 2nd order schemes.
\par In many situations, for example, in computations of the interactions of wave packets with localized potentials (barriers and especially wells) non-uniform meshes near them are very desirable.
The stability bounds and error estimates in $L^2$ and $H^1$ mesh norms for the Numerov-Crank-Nicolson scheme on non-uniform space meshes for the 1D time-dependent Schr\"{o}dinger equation has been recently derived by the second author in \cite{Z15}.
The analysis has turned out to be much more complicated than for the 2nd order schemes or any order finite element method in space \cite{DZZ09,ZZ12} due to the non-self-adjointness and non-positive definiteness of the Numerov averaging operator in general.

\par Moreover, the analysis needs a collection of additional not so simple conditions on meshes and the potential and leads to the non-uniform in time bounds in contrast to the absence of \textit{any conditions} and uniform bounds in \cite{DZZ09,ZZ12}.
Surprisingly, even for the zero potential, this collection contains the non-standard ``converse'' condition like
\begin{equation}
 h_\omega\leq c_0\tau,
\label{convcond}
\end{equation}
where $h_\omega$ is the mean space step, $\tau$ is the time step and $c_0>0$.

The aim of this paper is to demonstrate that this condition is nevertheless \textit{necessary} at least for some families of non-uniform meshes but \textit{any space norm}. We confine ourselves by the case of zero Dirichlet boundary condition much simpler than in \cite{DZZ09,ZZ12,Z15}.

\par We also present numerical results showing unacceptably wrong behavior of numerical solutions, namely, their dramatic mass non-conservation, when this condition is violated, i.e. $\tau$ decreases for fixed $h_\omega$.
\par Practically, the results of the paper lead to the conclusion that the Numerov-Crank-Nicolson scheme on non-uniform space meshes is not robust.
Unfortunately it seems that the similar situation can occur easily for compact schemes in general at the absence of clear and rigorous stability theory for them.
In this respect the alternative higher order finite element method in space looks much more attractive.
\section{\large Main results}
\label{sect1}
We consider the initial-boundary value problem for the 1D time-dependent Schr\"o\-din\-ger equation
\begin{gather}
 i\hbar\frac{\partial\psi}{\partial t}=
 -\frac{\hbar^2}{2m_0}\,\frac{\partial^2\psi}{\partial x^2}+V\psi\ \ \mbox{on}\ \ \Omega=(0,X),\ \ \mbox{for}\ \ t>0,
\label{eq0}\\[1mm]
 \psi|_{x=0,X}=0,\ \
 \left. \psi\right|_{t=0}=\psi^0(x)\ \ \mbox{on}\ \ \Omega.
\label{ic}
\end{gather}
Hereafter $i$ is the imaginary unit, $\hbar>0$ and $m_0>0$ are physical constants, the sought wave function $\psi=\psi(x,t)$ is complex-valued and $V=V(x)$ is the given real coefficient (potential).
Let $c_\hbar:=\hbar^2/(2m_0)$.
For this problem, the mass conservation property is well-known
\begin{gather}
 \|\psi(\cdot,t)\|_{L^2(\Omega)}^2=\|\psi^0\|_{L^2(\Omega)}^2\ \ \mbox{for}\ \ t\geq 0.
\label{masslaw}
\end{gather}

\par We define a non-uniform mesh $\overline{\omega}_{h}$ in $x$ on $[0,X]$ with the
nodes $0=x_0<\dots <x_J=X$ and the steps $h_j:=x_j-x_{j-1}$.
We utilize the backward and modified forward difference quotients
\begin{gather*}
 \overline{\partial}_xW:= \frac{W-W_-}{h},\ \
 \widehat{\partial}_xW:= \frac{W_+ -W}{\bar{h}},\ \
 \overline{s}_xW:= \frac{W_-+ W}{2}
\end{gather*}
with $W_{\pm j}:=W_{j\pm 1}$ for any $j$ and $\bar{h}_{j}:=(h_j+h_{j+1})/2$.

\par Let $H(\overline{\omega}_{h})$ be the space of the complex-valued functions $W$ on $\overline{\omega}_{h}$ such that $W_0=W_J=0$ equipped with a mesh counterpart of the complex $L_2(\Omega)$--inner product and the associated norm
\begin{gather*}
 (U,W)_{\omega_h}:=\sum_{j=1}^{J-1}U_jW_j^*\bar{h}_j,\ \ \|U\|_{\omega_h}:=(U,U)_{\omega_h}^{1/2},
 \end{gather*}
where $z^*$ denotes the complex conjugate for $z\in \mathbb{C}$.

\par We define the uniform mesh in $t$ with the nodes $t_m=m\tau$, $m\geq 0$, and the step $\tau>0$.
We use the backward difference quotient, the symmetric average and the backward shift in $t$
\[
 \overline{\partial}_tY:= \frac{Y-\check{Y}}{\tau},\ \
 \overline{s}_tY:= \frac{\check{Y}+Y}{2},\ \
 \check{Y}^m:=Y^{m-1}.
\]
\par We recall the three-point Numerov averaging operator
\[
 s_NW:=\frac{1}{12}\,\alpha W_{-}+\frac{5}{6}\,\gamma W+\frac{1}{12}\,\beta W_{+},\  \text{with}\ \
 \alpha=2-\frac{h_{+}^2}{h\bar{h}},\
 \gamma=1+\frac{(h_{+}-h)^2}{5hh_{+}},\
 \beta=2-\frac{h^2}{h_{+}\bar{h}};
\]
for derivations and equivalent forms of $\alpha$, $\beta$ and $\gamma$, see \cite{JIS84,radziunas2014B,Z15}.
It is easy to check that $(\alpha+10\gamma+\beta)/12=1$.
\par Recall that the natural property $\alpha_j\geqslant 0$ and $\beta_j\geqslant 0$ (not in use below) is valid only under the rather restrictive condition
$\frac{2}{\sqrt{5}+1}\leq\frac{h_{j+1}}{h_j}\leq\frac{\sqrt{5}+1}{2}$
noted in \cite{JIS84}.
For the uniform mesh $\overline{\omega}_h$, clearly $\alpha=\beta=\gamma=1$ and $s_N$ takes the most well-known form
$s_NW=(W_{-}+10W+W_{+})/12$.

\par We study the following Numerov-Crank-Nicolson scheme for problem \eqref{eq0}-\eqref{ic}
\begin{gather}
 i\hbar s_N\overline{\partial}_t\Psi^m=-c_\hbar\widehat{\partial}_x\overline{\partial}_x\overline{s}_t\Psi^m+s_N(V\overline{s}_t\Psi^m)\ \ \mbox{on}\ \ \omega_h,\ \ m\geq 1,
\label{dse}\\[1mm]
 \Psi_0^m=\Psi_J^m=0,\ \ m\geq 0,\ \
\Psi^0\ \ \mbox{is given on}\ \ \overline{\omega}_h.
\label{dse2}
\end{gather}
\par We are interested in analysis of the stability with respect to the initial data bound
\begin{gather}
 \|\Psi^m\|_h\leq C(1+\varkappa\tau)^m\|\Psi^0\|_h\ \ \text{for any}\ \ \Psi^0\in H(\overline{\omega}_h),\ \ m\geq 0,
\label{stabbound}
\end{gather}
with some $C>0$ and $\varkappa>0$ independent of the meshes, for $0<\tau\leq\tau_0$.
Here $\|\cdot\|$ is \textit{an arbitrary fixed norm} in $H(\overline{\omega}_h)$, in particular, $\|\cdot\|_h=\|\cdot\|_{\omega_h}$.
\par We introduce the following generalized eigenvalue problem
\begin{gather}
 -\widehat{\partial}_x\overline{\partial}_xW=\lambda s_NW,\ \ W\in H(\overline{\omega}_h),
\ \ W\not\equiv 0.
\label{eigprob}
\end{gather}
The operator $-\widehat{\partial}_x\overline{\partial}_x$ is self-adjoint and positive definite in $H(\overline{\omega}_h)$ but in general $s_N$ is not, for the non-uniform $\overline{\omega}_h$ (see \cite{Z15} for more details). Consequently complex eigenvalues $\lambda=\lambda_R+i\lambda_I$ with $\lambda_I\neq 0$ can exist, and they do exist for some meshes, see Section \ref{sect3}.
Clearly if $\{\lambda,W\}$ is an eigenpair of this problem, then $\{\lambda^*,W^*\}$ is an eigenpair too.
\par Let $\{\lambda,W\}$ be an eigenpair of problem \eqref{eigprob} and $\Psi^0=W$.
Then the function $\Psi^m=\Phi^mW$ solves equations \eqref{dse}-\eqref{dse2} with $V=0$ provided that
\[
 i\overline{\partial}_t\Phi^m=a\lambda\overline{s}_t\Phi^m,\ \ m\geq 1,\ \ \Phi^0=1,
\]
with $a:=c_\hbar/\hbar$, i.e.
\[
 \Phi^m=q_\tau^m(\lambda),\ \ q_\tau(\lambda)=\frac{1-i\frac{a\lambda\tau}{2}}{1+i\frac{a\lambda\tau}{2}}.
\]
For such a solution, the stability bound \eqref{stabbound} is equivalent to the spectral stability condition
\begin{gather}
 |q_\tau(\lambda)|\leq 1+\varkappa\tau.
\label{spectrcond}
\end{gather}

\par It is easy to calculate
\[
 |q_\tau(\lambda)|^2=1+\frac{4 \tilde{\lambda}_I\tau}{(\tilde{\lambda}_R\tau)^2+(\tilde{\lambda}_I\tau-1)^2},\ \
 \text{with}\ \ \tilde{\lambda}=\tilde{\lambda}_R+i\tilde{\lambda}_I:=\frac{a\lambda}{2}.
\]
Consequently, for $\lambda_I>0$ one can rewrite \eqref{spectrcond} as
\begin{gather}
 \frac{4}{\varkappa(2+\varkappa\tau)}+2\tau\leq\frac{|\tilde{\lambda}|^2}{\tilde{\lambda}_I}\tau^2
 +\frac{1}{\tilde{\lambda}_I}.
\label{stabcond 1}
\end{gather}
\par Let $\overline{\omega}_{h}^{\,0}$ be a fixed mesh with the nodes $0=x_0^0<\ldots<x_{J_0}^0=X$ (where $J_0\geq 2$) and the mean step $h_{\omega^0}=X/J_0$.
We construct the family of meshes $\overline{\omega}_{h}^{\,0,K}$, $K\geq 1$, on $\bar{\Omega}$ with the nodes
\begin{gather*}
 x_{2kJ_0+l}=2k\frac{X}{K}+\frac{x_l^0}{K}\ \ \text{for any}\ \ 0\leq l\leq J_0-1,\ \ 0\leq2k<K,
\\
 x_{(2k-1)J_0+l}=2k\frac{X}{K}-\frac{x_{J_0-l}^0}{K}\ \ \text{for any}\ \ 0\leq l\leq J_0-1,\ \ 1\leq 2k-1<K,
\end{gather*}
and $x_J=X$, with $J=J_0K$, and the mean step $h_\omega=X/J=h_{\omega^0}/K$.

\par We define the operator $\Pi_K: H(\overline{\omega}_h^{\,0})\to H(\overline{\omega}_h^{\,0,K})$ such that
\begin{gather*}
 \Pi_Kw_{2kJ_0+l}=w_l\ \ \text{for any}\ \ 0\leq l\leq J_0-1,\ \ 0\leq2k<K,
\\
 \Pi_Kw_{(2k-1)J_0+l}=-w_{J_0-l}\ \ \text{for any}\ \ 0\leq l\leq J_0-1,\ \ 1\leq 2k-1<K.
\end{gather*}
\begin{proposition}
\label{prop1}
Let $\{\lambda^0,w\}$ be an eigenpair of problem \eqref{eigprob} for $\overline{\omega}_h=\overline{\omega}_h^{\,0}$.
Then $\{\lambda^0K^2,\Pi_Kw\}$ is an eigenpair of problem \eqref{eigprob} for $\overline{\omega}_h=\overline{\omega}_h^{\,0,K}$.
\end{proposition}
\begin{proof}
A function $\widehat{\partial}_x\overline{\partial}_xW_j$ is scaled by a multiplier $\alpha^{-2}$ and $s_NW_j$ remains unchanged after scaling $h_j\to h_j/\alpha$ and $h_{j+1}\to h_{j+1}/\alpha$ with $\alpha>0$. Consequently
\begin{gather}
 -\widehat{\partial}_x\overline{\partial}_x\Pi_Kw_j=\lambda K^2 s_N\Pi_Kw_j\ \ \text{for}\ \ 1\leq j\leq J_0K-1,\ \ j\not\in\{kJ_0\}_{k=1}^{K-1}.
\label{eigprob 2}
\end{gather}
\par In addition we have
\[
 \widehat{\partial}_x\overline{\partial}_x\Pi_Kw_j=0,\ \ s_N\Pi_Kw_j=0\ \ \text{for}\ \ j\in \{kJ_0\}_{k=1}^{K-1}
\]
since $\Pi_Kw_{j-1}=-\Pi_Kw_{j+1}$, $h_j=h_{j+1}$ and $\Pi_Kw_j=0$ for such $j$.
Therefore equality \eqref{eigprob 2} holds for the remaining indexes $j\in\{kJ_0\}_{k=1}^{K-1}$ as well.
\end{proof}
Note that a similar trick has been recently applied in quite another context in \cite{ZZ17}.

\begin{proposition}
\label{prop2}
Let $\lambda^0=\lambda_R^0+i\lambda_I^0$ with $\lambda_I^0\neq 0$ be an eigenvalue of problem  \eqref{eigprob} for a non-uniform mesh $\overline{\omega}_h=\overline{\omega}_h^{\,0}$.
Then the following condition
\begin{gather}
 \frac{1}{\sqrt{\varkappa(2+\varkappa\tau_0)}}\leq c_1\frac{\tau}{h_\omega}+c_2h_\omega,
\label{nesscond}
\end{gather}
with $c_1>0$ and $c_2>0$ depending solely on $a\lambda^0h_{\omega^0}^2$,
is necessary for validity of the spectral condition \eqref{spectrcond} for the mesh $\overline{\omega}_h=\overline{\omega}_h^{\,0,K}$.
\end{proposition}
\begin{proof}
We can assume that $\lambda_I^0>0$.
Then owing to Proposition \ref{prop1} condition \eqref{stabcond 1} takes the form
\begin{gather} \label{est1}
 \frac{4}{\varkappa(2+\varkappa\tau)}+2\tau
 \leq a\frac{|\lambda^0|^2}{2\lambda_I^0}
 h_{\omega^0}^2\left(\frac{\tau}{h_\omega}\right)^2
 +\frac{2}{a\lambda_I^0h_{\omega^0}^2}h_\omega^2.
\end{gather}
It implies the necessary condition \eqref{nesscond}.
\end{proof}
In particular, for $2c_2\sqrt{\varkappa(2+\varkappa\tau_0)}h_\omega\leq 1$ condition \eqref{nesscond} leads to the ``converse'' condition
\eqref{convcond} with $c_0=2c_1\sqrt{\varkappa(2+\varkappa\tau_0)}$.
Since $\tau \to 0$ and $h_\omega \to 0$, asymptotically we get from \eqref{est1} the more precise constant
$c_0 = \sqrt{a \varkappa}h_{\omega^0}\vert \lambda^0 \vert \,/(2 \sqrt{\lambda_I^0})$.


%
\section{\large Numerical experiments}
\label{sect3}
We consider the Schr{\" o}dinger equation \eqref{eq0} in the simplest case $V(x)=0$, scale time $t$
so that $a=c_\hbar/\hbar=1$ and take the Gaussian wave packet
\begin{equation}
 \psi(x,t):=\sqrt{\frac{1}{1+4it}}\exp\frac{-(x-x_c)^2+ik(x-x_c-kt)}{1+4it}.
\label{eq:ex1}
\end{equation}
Its movement is simulated for $(x,t)\in [0,30]\times [0,1]$, $x_c=5$ and $k=4$ similarly to \cite{radziunas2014B}.
The initial and boundary data correspond to this exact solution; the boundary data are inhomogeneous but very small.

We have analyzed numerically the eigenvalue problem \eqref{eigprob} for small $J$.
For $J_0=14$, we have found the following simply arranged mesh on $[0,30]$ with the following scaled steps $\tilde{h}_j=1.5h_j$:
\[
\tilde{h}_{1} = \tilde{h}_{2} = \tilde{h}_{3} = 1,\ \tilde{h}_{4} =5,\ \tilde{h}_{5} =3,\ \tilde{h}_{6} =1,\ \tilde{h}_{7} =3,\ \tilde{h}_{8} =4,\
\tilde{h}_{j}=5\ (9\leq j\leq 14).
\]
The largest {in modulus} eigenvalue of problem \eqref{eigprob} is real
$\lambda = 9.91401$,
the next two eigenvalues are complex $\lambda = 2.57301 \pm  0.024621 i$ and the remaining eigenvalues are real again.

\par Then the above defined family of meshes $\overline{\omega}_{h}^{\,0,K}$, $K>1$, is applied in numerical experiments.
First we use the meshes with $J=560=40J_0$ and $J=1120=80J_0$.
In Fig.~\ref{figure1} the dynamics of $\|\Psi^m\|_{\omega_h}$ for $0\leq t_m\leq 1$ is shown for different time steps $\tau=1/M$.
One can see that the mass conservation law \eqref{masslaw} is violated dramatically as $M$ increases, i.e. $\tau$ decreases and violates condition \eqref{nesscond}.

\begin{figure}[ht]
\begin{center}
\begin{tabular}{cc}
\includegraphics[scale=0.3]{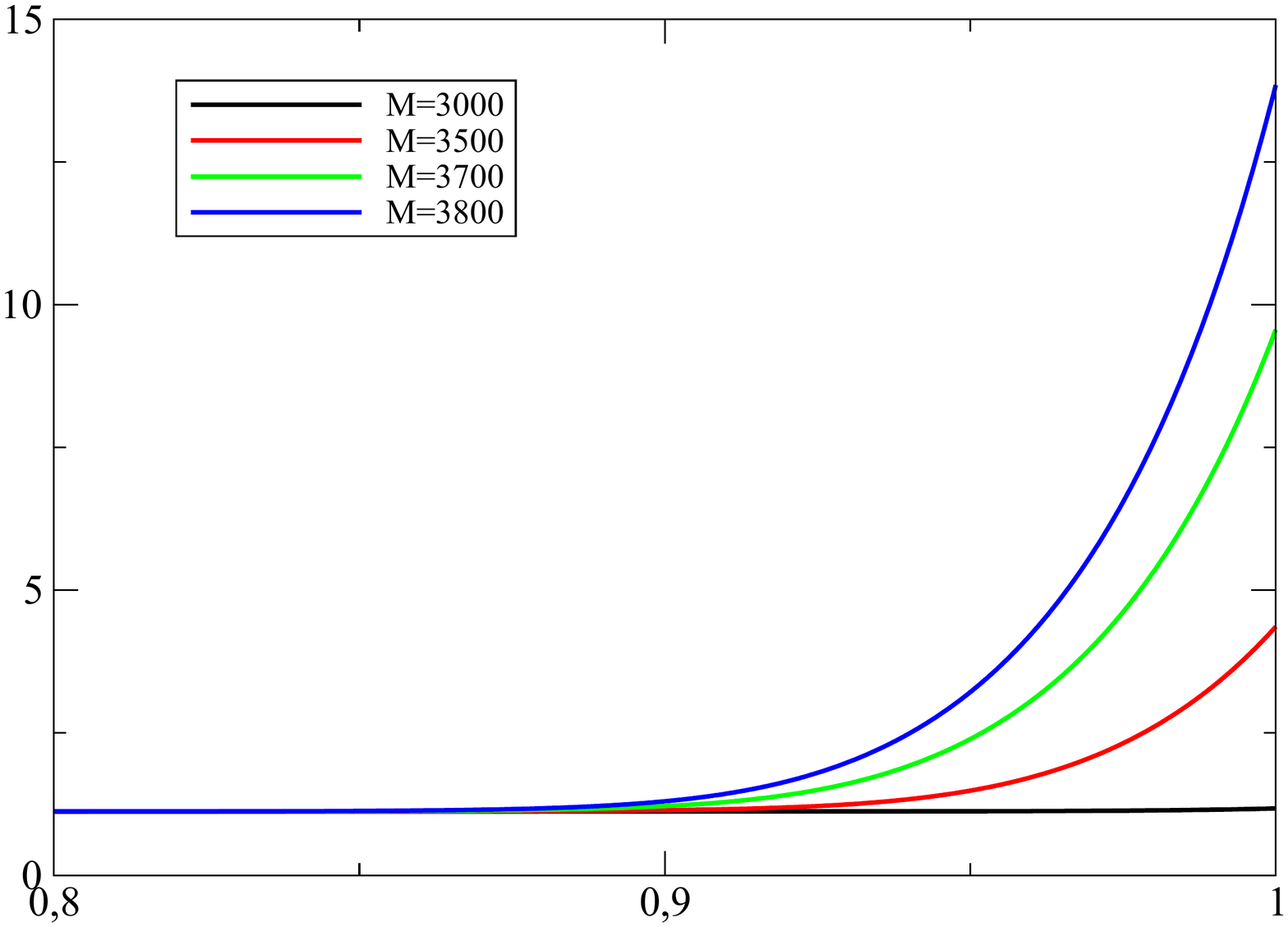} &
\includegraphics[scale=0.3]{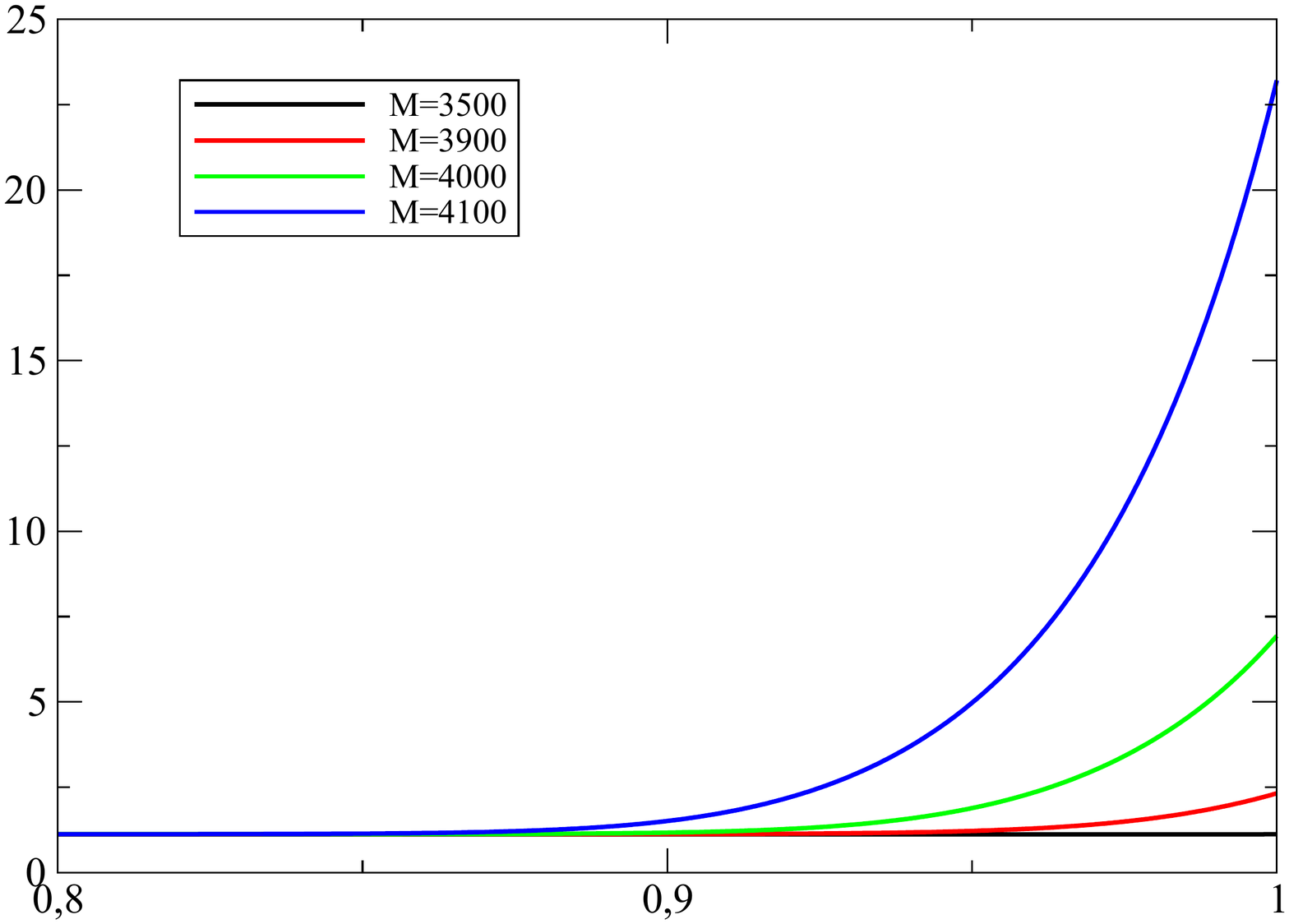} \\
a) $J=560$  & b) $J=1120$
\end{tabular}
\caption{The dynamics of $\|\Psi^m\|_{\omega_h}$ for
different numbers of time steps on $[0,1]$
}
\label{figure1}
\end{center}
\end{figure}

In the next numerical experiments we investigate the accuracy and the practical order that can be achieved for the Numerov-Crank-Nicolson scheme keeping in mind condition \eqref{nesscond}.
Let us define errors of the discrete solution in the maximum norm
and the convergence rates
\[
 err(J,M):=\max_{0\leq m\leq M}\max_{0\leq j\leq J}|\psi(x_j,t_m)-\Psi_j^m|,\ \
 p_k:=\log\frac{err(J_k,M_k)}{err(J_{k+1},M_{k+1})}\Big/\log\frac{J_{k+1}}{J_k},\ \ J_{k+1}>J_k.
\]
We choose the number of time steps $M_*=M(J)$ such that the error $err(J,M_*)$
is minimal up to the local search with a step equal to 100.
Results are presented in Table~\ref{table1}.
\begin{table}[ht]
\begin{center}
\caption{The errors $err(J,M_*)$ and the corresponding convergence rates $p$}
\label{table1}
\vspace{2mm}
\begin{tabular}{llccc}
\hline\noalign{\smallskip}
$k$ & $J$ & $M_*$ & $err(J,M_*)$ & $ p$ \\
\noalign{\smallskip}
\hline
\noalign{\smallskip}
0& 560  & 2000  & 1.116e-3 & ---     \\
1& 840  & 2300  & 3.40e-4  & 3.027    \\
2& 1120 & 2900  & 1.45e-4  & 2.962    \\
3& 1680 & 4000  & 5.57e-5  & 2.360    \\
\hline
\end{tabular}
\end{center}
\end{table}
\par We can observe that the convergence rate is close to the 3rd order for smaller
values of $J$ but it reduces down as $J$ increases, and thus the scheme asymptotically tends to lose its main higher
order property.
\newpage
\small{

}
\end{document}